\documentclass{amsart}
\usepackage[latin1]{inputenc}
\usepackage[T1]{fontenc}
\usepackage{textcomp}
\usepackage[english]{babel}
\usepackage{amsmath}
\usepackage{amsfonts}
\usepackage{amssymb}
\usepackage{amsopn}
\usepackage{amscd}
\usepackage{amsthm}
\usepackage{mathrsfs}
\usepackage{graphicx}
\usepackage{enumerate}
\theoremstyle{plain}
\newtheorem{theorem}{Theorem}[section]
\newtheorem{lemma}[theorem]{Lemma}
\newtheorem{prop}[theorem]{Proposition}
\newtheorem{corollary}[theorem]{Corollary}

\theoremstyle{definition}
\newtheorem{definition}{Definition}[section]

\theoremstyle{remark}
\newtheorem*{remark}{Remark}

\newtheorem*{acknowledgements}{Acknowledgements}

\newcommand{\R}{\mathbb{R}}

\newcommand{\Pro}{\mathcal{P}}
\newcommand{\pr}{\mathbb{P}}

\newcommand{\bt}{\mathbf{T}}

\title[Mean speed of convergence in Wasserstein distance]{On the mean speed of convergence of empirical and occupation measures in Wasserstein distance}

  \author{Emmanuel Boissard}
  \author{Thibaut Le Gouic}

  \address{Universit\'e Paul Sabatier}

\begin{document}

\begin{abstract}
In this work, we provide non-asymptotic bounds for the average speed of convergence of the empirical measure in the law of large numbers, in Wasserstein distance.
We also consider occupation measures of ergodic Markov chains. One motivation is the approximation of a probability measure by finitely supported measures (the quantization problem).
It is found that rates for empirical or occupation measures match or are close to previously known optimal quantization rates in several cases. This is notably highlighted in the example 
of infinite-dimensional Gaussian measures.
\end{abstract}

\date{\today}

\maketitle

\section{Introduction}

This paper is concerned with the rate of convergence in Wasserstein distance for the so-called \emph{empirical law of large numbers} :
let $(E, d, \mu)$ denote a measured Polish space, and let

\begin{equation}
L_n = \frac{1}{n} \sum_{i = 1}^{n} \delta_{X_i} 
\end{equation}

 denote the empirical measure
associated with the i.i.d. sample $(X_i)_{1 \leq i \leq n}$ of law $\mu$, then with probability 1, $L_n \rightharpoonup \mu$ as $n \rightarrow + \infty$ 
(convergence is understood in
the sense of the weak topology of measures). This theorem is also known as Glivenko-Cantelli theorem and is due in this form to Varadarajan
\cite{varadarajan1958convergence}.

For $1 \leq p < + \infty$, the \emph{$p$-Wasserstein distance} is defined on the set $\Pro_p(E)^2$ of couples of measures with a finite $p$-th moment by

\begin{equation*}
 W_p^p (\mu, \nu) = \inf_{ \pi \in \Pro(\mu, \nu)} \int d^p(x, y) \pi(dx, dy)
\end{equation*}

where the infimum is taken on the set $\Pro(\mu, \nu)$ of probability measures with first, resp. second, marginal $\mu$, resp. $\nu$. 
This defines a metric on $\Pro_p$, and convergence in this metric is equivalent to weak convergence plus convergence of the moment of order $p$.
These metrics, and more generally the Monge transportation problem from which they originate, have played a prominent role in several areas of 
probability, statistics and the analysis of P.D.E.s : for a rich account, see C. Villani's St-Flour course \cite{optimal_transport_villani}.

Our purpose is to give bounds on the mean speed of convergence in $W_p$ distance for the Glivenko-Cantelli theorem, i.e. bounds for the convergence
$\mathbb{E}(W_p( L_n, \mu)) \rightarrow 0$.
Such results are desirable notably in view of numerical and statistical applications~: indeed, the approximation
of a given probability measure by a measure with finite support in Wasserstein distance is a topic that appears in various guises in the literature, 
see for example \cite{graf2000foundations}.
The first motivation for this work was to extend the results obtained by
F. Bolley, A. Guillin and C. Villani \cite{quantit_conc_ineq} in the case of variables with support in $\R^d$. 
As in this paper, we aim to produce bounds that are non-asymptotic and effective (that is with explicit constants), in order to achieve practical relevance.

We also extend the investigation to the convergence of occupation measure for suitably ergodic Markov chains : again, we have practical applications in mind,
as this allows to use Metropolis-Hastings-type algorithms to approximate an unknown measure (see \ref{subsection_markov_chains} for a discussion of this).

There are many works in statistics devoted to convergence rates in some metric associated with the weak convergence of measures, see 
e.g. the book of A. Van der Vaart and J. Wellner \cite{van1996weak}. 
Of particular interest for us is R.M. Dudley's article \cite{dudley1969speed}, see Remark \ref{remark_dudley}.

Other works have been devoted to convergence of empirical measures in Wasserstein distance, we quote some of them. 
Horowitz and Karandikar 
\cite{horowitz1994mean} gave a bound for the rate of convergence of $\mathbb{E}[ W_2^2(L_n, \mu)]$ to $0$ for general measures supported in $\R^d$ under a moment condition.
M. Ajtai, J. Komlos and G. Tusnady \cite{ajtai1984optimal} and M.Talagrand \cite{talagrand1992matching} studied the related problem of the average cost of matching two
i.i.d. samples from the uniform law on the unit cube in dimension $d \geq 2$. 
This line of research was pushed further, among others, by V. Dobri\'c and J.E. Yukich \cite{dobric1995exact} or F. Barthe and C. Bordenave \cite{barthe2011combinatorial}
(the reader may refer to this last paper for an up-to-date account of the Euclidean matching problem). These papers give a sharp result for
measures in $\R^d$, with an improvement both over \cite{horowitz1994mean} and \cite{quantit_conc_ineq}.
In the case $\mu \in \Pro(\R)$, del Barrio, Gin\'{e} and Matran \cite{del1999central} obtain a central limit theorem for $W_1(L_n, \mu)$ under the condition that
$\int_{- \infty}^{+ \infty} \sqrt{F(t)(1 - F(t))} dt < + \infty$ where $F$ is the cumulative distribution function (c.d.f.) of $\mu$.
In the companion paper \cite{boissard2011exact}, we investigate the case of the $W_1$ distance by using the dual expression of the $W_1$ transportation cost
by Kantorovich and Rubinstein, see therein for more references.

Before moving on to our results, we make a remark on the scope of this work.
Generally speaking, the problem of convergence of $W_p(L_n, \mu)$ to $0$ can be divided in two separate questions : 

\begin{itemize}
 \item the first one is to estimate the \emph{mean rate of convergence}, that is the convergence rate of $\mathbb{E} [ W_p(L_n, \mu) ]$,
 \item while the second one is to study the concentration properties of $W_p(L_n, \mu)$ around its mean, that is to find bounds on the quantities

\begin{equation*}
 \pr (W_p(L_n, \mu) - \mathbb{E} [W_p(L_n, \mu) ] \geq t).
\end{equation*}
\end{itemize}

Our main concern here is the first point. The second one can be dealt with by techniques of measure concentration. 
We will elaborate on this in the case of Gaussian measures (see Appendix \ref{appendix_gaussian}), but not in general.
However, this is a well-trodden topic, and some results are gathered in \cite{boissard2011exact}.

\begin{acknowledgements}
 We thank Patrick Cattiaux for his advice and careful reading of preliminary versions, and Charles Bordenave for introducing us to his work \cite{barthe2011combinatorial}
and connected works.
\end{acknowledgements}

\subsection{Main result and first consequences}

\begin{definition}
For $X \subset E$, the covering number of order $\delta$ for $X$, denoted by $N(X, \delta)$, is defined as
the minimal $n \in \mathbb{N}$ such that there exist $x_1, \ldots, x_n$ in $X$ with

\begin{equation*}
 X \subset \bigcup_{j = 1}^n B(x_i, \delta).
\end{equation*}

\end{definition}

Our main statement is summed up in the following proposition.

\begin{prop} \label{main_prop_independent_case}

Choose $t > 0$. Let $\mu \in \Pro(E)$
with support included in $X \subset E$ with finite diameter $d$ such that $N(X, t) < + \infty$.
We have the bound :

\begin{equation*}
 \mathbb{E}(W_p(L_n, \mu)) \leq c \left( t + n^{-1/2p} \int_{t}^{d/4} N(X, \delta)^{1/2p} d \delta \right).
\end{equation*}

with $c \leq 64/3$.

\end{prop}

\begin{remark} \label{remark_dudley}
Proposition \ref{main_prop_independent_case} is related in spirit and proof to the results of R.M. Dudley \cite{dudley1969speed} in the case
of the bounded Lipschitz metric

\begin{equation*}
 d_{BL}( \mu, \nu) = \inf_{f 1-\text{Lip}, |f| \leq 1} \int f d ( \mu - \nu).
\end{equation*}

The analogy is not at all fortuitous : indeed, the bounded Lipschitz metric is linked to the $1$-Wasserstein distance 
via the well-known Kantorovich-Rubinstein dual definition of $W_1$ : 

\begin{equation*}
 W_1( \mu, \nu) = \inf_{f 1-\text{Lip}} \int f d (\mu - \nu).
\end{equation*}

The analogy stops at $p = 1$ since there is no representation of $W_p$ as an empirical process for $p > 1$ 
(there is, however, a general dual expression of the transport cost). In spite of this, the technique of proof in \cite{dudley1969speed} proves
useful in our case, and the technique of using a sequence of coarser and coarser partitions is at the heart of many later results, notably in
the literature concerned with the problem of matching two independent samples in Euclidean space, see e.g. \cite{talagrand1992matching} or the
recent paper \cite{barthe2011combinatorial}.

\end{remark}

We now give a first example of application, under an assumption that the underlying metric space is of finite-dimensional type in some sense.
More precisely, we assume that there exist $k_E > 0$, $\alpha > 0$ such that

\begin{equation} \label{assumption_dimension_independent_case}
 N(E, \delta) \leq k_E (\text{Diam } E/\delta)^\alpha.
\end{equation}

Here, the parameter $\alpha$ plays the role of a dimension.

\begin{corollary} \label{thm_independent_case_finite_dim}
 Assume that $E$ satisfies (\ref{assumption_dimension_independent_case}), and that $\alpha > 2p$. With notations as earlier, the following holds :

\begin{equation*}
 \mathbb{E} [W_p (L_n, \mu)] \leq c \frac{\alpha}{\alpha - 2p} \, \text{Diam } E \, k_E^{1/\alpha} n^{-1/\alpha}
\end{equation*}

with $c \leq 64/3$.
\end{corollary}

\begin{remark}
 In the case of measures supported in $\R^d$, this result is neither new nor fully optimal. For a sharp statement in this case, the reader may refer to \cite{barthe2011combinatorial}
and references therein. However, we recover at least the exponent of $n^{-1/ d}$
which is sharp for $d \geq 3$, see \cite{barthe2011combinatorial} for a discussion.
And on the other hand, Corollary \ref{thm_independent_case_finite_dim} extends to more general metric spaces of finite-dimensional type, for example manifolds.
\end{remark}

As opposed to Corollary \ref{thm_independent_case_finite_dim}, our next result is set in an infinite-dimensional framework.

\subsection{An application to Gaussian r.v.s in Banach spaces} \label{subsection_gaussian}

We apply the results above to the case where $E$ is a separable Banach space with norm $\|.\|$, 
and $\mu$ is a centered Gaussian random variable with values in $E$, meaning that the image of
$\mu$ by every continuous linear functional $f \in E^*$ is a centered Gaussian variable in $\R$.
The couple $(E, \mu)$ is called a (separable) Gaussian Banach space.

Let $X$ be a $E$-valued r.v. with law $\mu$, and define the weak variance of $\mu$ as 

\begin{equation*}
 \sigma = \sup_{f \in E^*, \, |f| \leq 1} \left( \mathbb{E} f^2(X) \right)^{1/2}.
\end{equation*}

The small ball function of a Gaussian Banach space $(E, \mu)$ is the function

\begin{equation*}
 \psi(t) = - \log \mu(B(0, t)).
\end{equation*}

We can associate to the couple $(E, \mu)$ their Cameron-Martin Hilbert space $H \subset E$, see e.g. \cite{ledoux1996isoperimetry}
for a reference. It is known that the small ball function has deep links with the covering numbers of the unit ball of $H$,
see e.g. Kuelbs-Li \cite{kuelbs1993metric} and Li-Linde \cite{li1999approximation}, as well as with the 
approximation of $\mu$ by measures with finite support in Wasserstein distance (the quantization or optimal quantization problem), 
see Fehringer's Ph.D. thesis \cite{fehringer2001kodierung}, Dereich-Fehringer-Matoussi-Scheutzow \cite{dereich2003link}, Graf-Luschgy-Pag\`{e}s
\cite{graf2003functional}.

We make the following assumptions on the small ball function :

\begin{enumerate}
 \item \label{cond_small_ball_1} there exists $\kappa > 1$ such that $\psi(t) \leq \kappa \psi(2t)$ for $0 < t \leq t_0$,
 \item \label{cond_small_ball_2} for all $\varepsilon > 0$, $n^{-\varepsilon} = o( \psi^{-1} ( \log n) )$.
\end{enumerate}

Assumption (\ref{cond_small_ball_2}) implies that the Gaussian measure is genuinely infinite dimensional : indeed, in the case when
$\text{dim } K < + \infty$, the measure is supported in a finite-dimensional Banach space, and in this case the small ball function behaves as $\log t$.

\begin{theorem} \label{theorem_gaussian_vector}

Let $(E, \mu)$ be a Gaussian Banach space with weak variance $\sigma$ and small ball function $\psi$. Assume that Assumptions
(\ref{cond_small_ball_1}) and (\ref{cond_small_ball_2}) hold.

Then there exists a universal constant $c$ such that for all

\begin{equation*}
n \geq (6 + \kappa) (\log 2 \vee \psi(1) \vee \psi(t_0/2) \vee 1/ \sigma^2), 
\end{equation*}

the following holds :

\begin{equation} \label{result_thm_gaussian_1}
 \mathbb{E}(W_2( L_n, \mu)) \leq c \left[  \psi^{-1}( \frac{1}{6 + \kappa} \log n) + \sigma n^{-1/[4(6 + \kappa)]} \right].
\end{equation}

In particular, there is a $C = C(\mu)$ such that

\begin{equation} \label{result_thm_gaussian_2}
 \mathbb{E}(W_2 (L_n, \mu)) \leq C \psi^{-1}( \log n).
\end{equation}

Moreover, for $\lambda > 0$,

\begin{equation} \label{result_thm_gaussian_3}
 W_2(L_n, \mu) \leq (C + \lambda) \psi^{-1}( \log n) \text{ with probability } 1 - \exp -n \psi^{-1}( \log n) \frac{\lambda^2}{2 \sigma^2}.
\end{equation}

\end{theorem}

\begin{remark}
 Note that the choice of $6 + \kappa$ is not particularly sharp and may likely be improved.
\end{remark}

In order to underline the interest of the result above, we introduce some definitions from optimal quantization.
For $n \geq 1$ and $1 \leq r < + \infty$, define the optimal quantization error at rate $n$ as

\begin{equation*}
 \delta_{n, r} (\mu) = \inf_{ \nu \in \Pro_n} W_r( \mu, \nu)
\end{equation*}

where the infimum runs on the set $\Pro_n$ of probability measures with finite support of cardinal bounded by $n$.
Under some natural assumptions, the upper bound of (\ref{result_thm_gaussian_3}) is matched by a lower bound 
for the quantization error. 
Theorem 3.1 in \cite{dereich2003link} states the following : if for every $0 < \zeta < 1$, 

\begin{equation*}
\mu( (1- \zeta) \varepsilon B ) = o (\mu ( \varepsilon B) ) \text{ as } \varepsilon \rightarrow 0, 
\end{equation*}

then 

\begin{equation*}
 \delta_{n, r} \gtrsim \psi^{-1} ( \log n )
\end{equation*}

(where $a_n \gtrsim b_n$ means $\liminf a_n / b_n \geq 1$).

In the terminology of quantization, Theorem \ref{theorem_gaussian_vector} states that the empirical measure is a rate-optimal quantizer with high probability (under
some assumptions on the small ball function). This is of practical interest, since obtaining the empirical measure is only as difficult as simulating an instance
of the Gaussian vector, and one avoids dealing with computation of appropriate weights in the approximating discrete measure.

We leave aside the question of determining the sharp asymptotics for the average error $\mathbb{E}(W_2(L_n, \mu))$, that is of
finding $c$ such that $\mathbb{E} (W_2 (L_n, \mu)) \sim c \psi^{-1}( \log n)$. Let us underline that the corresponding question for
quantizers is tackled for example in \cite{luschgy2004sharp}.

\subsection{The case of Markov chains} \label{subsection_markov_chains}

We wish to extend the control of the speed of convergence to weakly dependent sequences, such as rapidly-mixing Markov chains.
There is a natural incentive to consider this question : there are cases when one does not know hom to sample from a
given measure $\pi$, but a Markov chain with stationary measure $\pi$ is nevertheless available for simulation. This is the basic
set-up of the Markov Chain Monte Carlo framework, and a very frequent situation, even in finite dimension.

When looking at the proof of Proposition \ref{main_prop_independent_case}, it is apparent that 
the main ingredient missing in the dependent case is the argument following
(\ref{eq_bound_empirical}), i.e. that whenever $A \subset X$ is measurable, $n L_n(A)$ follows a binomial law with parameters $n$ and $\mu(A)$, 
and this must be remedied in some way.
It is natural to look for some type of quantitative ergodicity property of the chain, expressing almost-independence of $X_i$ and $X_j$ in the long range
($|i-j|$ large). 

We will consider decay-of-variance inequalities of the following form :

\begin{equation} \label{decay_of_variance}
 \text{Var}_\pi P^n f \leq C \lambda^n \text{Var}_\pi f.
\end{equation}

In the reversible case, a bound of the type of (\ref{decay_of_variance}) is ensured by Poincar\'e or spectral gap inequalities. 
We recall one possible definition in the discrete-time Markov chain setting.

\begin{definition}

Let $P$ be a Markov kernel with \emph{reversible} measure $\pi \in \Pro(E)$. We say that a Poincar\'e inequality with constant $C_P > 0$ holds if

\begin{equation} \label{def_poincare_discrete}
 \text{Var}_\pi f \leq C_P \int f (I - P^2)f d \pi
\end{equation}

for all $f \in L^2(\pi)$.

If (\ref{def_poincare_discrete}) holds, we have

\begin{equation*}
 \text{Var}_\pi P^n f \leq \lambda^n \text{Var}_\pi f
\end{equation*}

with $\lambda = (C_P-1)/C_P$.
\end{definition}

More generally, one may assume that we have a control of the decay of the variance in the following form :

\begin{equation} \label{decay_weak}
 \text{Var}_\pi P^n f \leq C \lambda^n \| f - \int f d \pi \|_{L^p}.
\end{equation}

As soon as $p > 2$, these inequalities are weaker than (\ref{decay_of_variance}). Our proof would be easily adaptable to this weaker decay-of-variance setting. 
We do not provide a complete statement of this claim.

For a discussion of the links between Poincar\'e inequality and other notions of weak dependence (e.g. mixing coefficients), see the recent paper \cite{cattiaux2011central}.

For the next two theorems, we make the following dimension assumption on $E$ : there exists $k_E > 0$ and $\alpha > 0$ such that for all $X \subset E$ with finite diameter,

\begin{equation} \label{dimension_assumption_markov}
N(X, \delta) \leq k_E (\text{Diam } X /\delta)^\alpha.
\end{equation}

The following theorem is the analogue of Corollary \ref{thm_independent_case_finite_dim} under the assumption that the Markov chain satisfies a decay-of-variance inequality.

\begin{theorem} \label{thm_markov}
Assume that $E$ has finite diameter $d > 0$ and (\ref{dimension_assumption_markov}) holds.
Let $\pi \in \Pro(E)$, and let $(X_i)_{i \geq 0}$ be a $E$-valued Markov chain with initial law $\nu$ such that $\pi$ is its unique invariant probability.
Assume also that (\ref{decay_of_variance}) holds for some $C > 0$ and $\lambda < 1$.

Then if $2p > \alpha(1 + 1/r)$ and $L_n$ denotes the occupation measure $1/n \sum_{i = 1}^n \delta_{X_i}$, the following holds :

\begin{equation*}
 \mathbb{E}_\nu \left[ W_p(L_n, \pi) \right] \leq c \frac{\alpha(1 + 1/r)}{\alpha(1 + 1/r) - 2p} k_E^{1/ \alpha} d \left( \frac{ C \| \frac{d \nu}{ d \pi} \|_{r}}{(1 - \lambda)n} \right)^{1/[\alpha(1 + 1/r)]}
\end{equation*}

for some universal constant $c \leq 64/3$.

\end{theorem}

The previous theorem has the drawback of assuming that the state space has finite diameter. This can be circumvented, for example by truncation arguments. 
Our next theorem is an extension to the unbounded case under some moment conditions on $\pi$. The statement and the proof involve more technicalities than Theorem \ref{thm_markov}, so we 
separate the two in spite of the obvious similarities.

\begin{theorem} \label{thm_markov_unbounded}

Assume that (\ref{dimension_assumption_markov}) holds.
Let $\pi \in \Pro(E)$, and let $(X_i)_{i \geq 0}$ be a $E$-valued Markov chain with initial law $\nu$ such that $\pi$ is its unique invariant probability.
Assume also that (\ref{decay_of_variance}) holds for some $C > 0$ and $\lambda < 1$.
Let $x_0 \in E$ and for all $\theta \geq 1$, denote $M_\theta = \int d(x_0, x)^\theta d \pi$.
Fix $r$ and $\zeta > 1$
and assume $2p > \alpha(1 + 1/r)(1 + 1/\zeta)$. 

There exist two numerical constant $C_1(p, r, \zeta)$ and $C_2(p, r, \zeta)$ only depending on $p$, $r$ and $\zeta$ such that whenever

\begin{equation*}
  \frac{ C \| \frac{d \nu}{ d \pi} \|_{r}}{(1 - \lambda)n} \leq C_1(p, r, \zeta),
\end{equation*}

the following holds :

\begin{equation*}
 \mathbb{E}_\nu \left[ W_p(L_n, \pi) \right] \leq C_1(p, r, \zeta) K(\zeta)  \left( \frac{ C \| \frac{d \nu}{ d \pi} \|_{r}}{(1 - \lambda)n} \right)^{1/[\alpha(1 + 1/r)(1 + 1/\zeta)]}
\end{equation*}

where 

\begin{equation*}
 K(\zeta) = \frac{m_\zeta}{m_p^{\zeta/p}} \vee \frac{m_{\zeta + 2p}}{m_p^{1 + \zeta/p}} \vee k_E^{1/2p(1+1/r)} \frac{2p}{\alpha(1 + 1/r)} m_p^{\alpha/(2p^2)(1+1/r)}.
\end{equation*}

\end{theorem}

\section{Proofs in the independent case}

\begin{lemma} \label{lemma_partitions}

 Let $X \subset E$, $s > 0$ and $u, v \in \mathbb{N}$ with $u < v$. Suppose that $N(X, 4^{-v} s) < + \infty$. For $u \leq j \leq v$, there exist integers 

\begin{equation}
 m(j) \leq N(X, 4^{-j} s)
\end{equation}

and non-empty subsets $X_{j, l}$ of $X$, $u \leq j \leq v$, $1 \leq l \leq m(j)$, such that the sets $X_{j, l}$ $1 \leq l \leq m(j)$ 
satisfy

\begin{enumerate}
 \item \label{cond_1} for each $j$, $(X_{j, l})_{1 \leq l \leq m(j)}$ is a partition of $X$,
 \item \label{cond_2} $\text{Diam } X_{j, l} \leq 4^{-j + 1} s$,
 \item \label{cond_3} for each $j > u$, for each $1 \leq l \leq m(j)$ there exists $1 \leq l' \leq m(j-1)$ such that $X_{j, l} \subset X_{j-1, l'}$.
\end{enumerate}
 
In other words, the sets $X_{j, l}$ form
a sequence of partitions of $X$ that get coarser as $j$ decreases (tiles at the scale $j-1$ are unions of tiles at the scale $j$).
\end{lemma}

\begin{proof}
 We begin by picking a set of balls $B_{j, l} = B(x_{j,l}, 4^{-j}s)$ with $u \leq j \leq v$ and $1 \leq l \leq N(X, 4^{-j}s)$, such that
for all $j$, 

\begin{equation*}
X \subset \bigcup_{l = 1}^{N(X, 4^{-j}s)} B_{j, l}.
 \end{equation*}

Define $X_{v, 1} = B_{v, 1}$, and successively set $X_{v, l} = B_{v, l} \setminus X_{v, l-1}$. Discard the possible empty sets and relabel the existing sets accordingly.
We have obtained the finest partition, obviously satisfying conditions (\ref{cond_1})-(\ref{cond_2}).

Assume now that the sets $X_{j, l}$ have been built for $k+1 \leq j \leq v$. Set $X_{k, 1}$ to be the reunion of all $X_{k+1, l'}$ such that
$X_{k+1, l'} \cap B_{k, 1} \neq \emptyset$. Likewise, define by induction on $l$ the set $X_{k, l}$ as the reunion of all $X_{k+1, l'}$ such that
$X_{k+1, l'} \cap B_{k, l} \neq \emptyset$ and $X_{k+1, l'} \nsubseteq X_{k, p}$ for $1 \leq p < l$. 
Again, discard the possible empty sets and relabel the remaining tiles. It is readily checked that the sets obtained satisfy assumptions (\ref{cond_1}) and (\ref{cond_3}).
We check assumption (\ref{cond_2}) : let $x_{k,l}$ denote the center of $B_{k, l}$ and let $y \in X_{k+1, l'} \subset X_{k,l}$. We have

\begin{equation*}
 d(x_{k,l}, y) \leq 4^{-k}s + \text{Diam } X_{k+1, l'} \leq 2 \times 4^{-k}s,
\end{equation*}

thus $\text{Diam } X_{k, l} \leq 4^{-k + 1}s$ as desired.

\end{proof}

Consider as above a subset $X$ of $E$ with finite diameter $d$, and assume that $N(X, 4^{-k}d) < + \infty$. 
Pick a sequence of partitions $(X_{j, l})_{1 \leq l \leq m(j)}$ for $1 \leq j \leq k$, as per Lemma \ref{lemma_partitions}.
For each $(j, l)$ choose a point $x_{j,l} \in X_{j, l}$.
Define the set of points of level $j$ as the set $L(j) =  \{x_{j, l} \}_{1 \leq l \leq m(j)}$. 
Say that $x_{j', l'}$ is an ancestor of $x_{j, l}$ if $X_{j, l} \subset X_{j', l'}$ : we will denote this relation by $(j', l') \rightarrow (j, l)$.

The next two lemmas study the cost of transporting a finite measure $m_k$ to another measure $n_k$ when these measures have
support in $L(k)$. The underlying idea is that we consider the finite metric space formed by the points $x_{j, l}$, $1 \leq j \leq k$,
as a \emph{metric tree}, where points are connected to their ancestor at the previous level, and we consider
the problem of transportation between two masses at the leaves of the tree. The transportation algorithm we consider
consists in allocating as much mass as possible at each point, then moving the remaining mass up one level in the tree, and iterating
the procedure.

A technical warning : please note that the transportation cost is usually defined between two probability measures ; however there is no difficulty in
extending its definition to the transportation between two finite measures of equal total mass, and we will freely
use this fact in the sequel.

\begin{lemma} \label{lemma_induction_cost_intermediate}
Let $m_j$, $n_j$ be measures with support in $L_j$. Define the measures $\tilde{m}_{j-1}$ and $\tilde{n}_{j-1}$ on 
$L_{j-1}$ by setting 

\begin{align}
 \label{def_tilde_m} \tilde{m}_{j-1} (x_{j-1, l'}) &  = \sum_{(j-1, l') \rightarrow (j, l)} (m_{j}(x_{j, l}) - n_{j}(x_{j, l})) \wedge 0, \\
 \label{def_tilde_n} \tilde{n}_{j-1} (x_{j-1, l'}) & = \sum_{(j-1, l') \rightarrow (j, l)} (n_{j}(x_{j, l}) - m_{j}(x_{j, l})) \wedge 0.
\end{align}

The measures $\tilde{m_{j-1}}$ and $\tilde{n_{j-1}}$ have same mass, so the transportation cost between them may be defined. 
Moreover, the following bound holds :

\begin{equation}
 W_p( m_j, n_j) \leq  2 \times 4^{-j + 2} d  \|m_j - n_j\|_{TV}^{1/p} + W_p ( \tilde{m}_{j-1} , \tilde{n}_{j-1} ).
\end{equation}

\end{lemma}

\begin{proof}
Set $m_j \wedge n_j (x_{j, l}) =  m_j (x_{j, l}) \wedge n_j (x_{j, l})$. By the triangle inequality,

\begin{align*}
 W_p( m, n) \leq & W_p( m_j, m_j \wedge n_j) + \tilde{m}_{j-1} + W_p(m_j \wedge n_j + \tilde{m}_{j-1}, m_j \wedge n_j + \tilde{n}_{j-1}) \\
{} & + W_p(m_j \wedge n_j + \tilde{n}_{j - 1}, n_j). 
\end{align*}

We bound the term on the left. Introduce the transport plan $\pi_m$ defined by

\begin{align*}
\pi_m (x_{j ,l}, x_{j, l}) & =  m_j \wedge n_j (x_{j, l}), \\
\pi_m (x_{j ,l}, x_{j-1, l'}) & =  (m_j (x_{j, l}) - n_j(x_{j, l}))_ {+} \text { when } (j-1, l') \rightarrow (j, l).
\end{align*}

The reader can check that $\pi_m \in \Pro( m_j, m_j \wedge n_j + \tilde{m}_{j-1})$. Moreover,

\begin{align*}
 W_p( m_j, \tilde{m_{j-1}} ) & \leq \left( \int d^p(x, y) \pi_m(d x, d y) \right)^{1/p} \\
{} & \leq 4^{-j + 2} d \left( \sum_{ l = 1}^{m(j)} (m_j (x_{j, l}) - n_j(x_{j, l}))_ {+} \right)^{1/p}.
\end{align*}

Likewise,

\begin{equation*}
 W_p( n_j, m_j \wedge n_j + \tilde{n_{j-1}} ) \leq 4^{-j + 2} d \left( \sum_{ l = 1}^{m(j)} (n_j (x_{j, l}) - m_j(x_{j, l}))_ {+} \right)^{1/p}.
\end{equation*}

As for the term in the middle, it is bounded by $W_p ( \tilde{m}_{j-1} , \tilde{n}_{j-1} )$.
Putting this together 
and using the inequality $x + y \leq 2^{1 - 1/p} (x^p + y^p)^{1/p}$,, we get

\begin{equation*}
 W_p( m_j, n_j) \leq  2^{1 - 1/p} 4^{-j + 2} d  \left( \sum_{ l = 1}^{m(j)} |m_j (x_{j, l}) - n_j(x_{j, l})| \right)^{1/p} + W_p ( \tilde{m}_{j-1} , \tilde{n}_{j-1} ).
\end{equation*}

\end{proof}

\begin{lemma} \label{lemma_induction_cost}
 Let $m_j$, $n_j$ be measures with support in $L_j$. Define for $1 \leq j' < j$ the measures $m_j'$, $n_j'$ with support in $L_j'$ by

\begin{equation} \label{def_measures_above}
 m_{j'}(x_{j', l'})   = \sum_{(j', l') \rightarrow (j, l)} m_{j}(x_{j, l}) , \quad n_{j'}(x_{j', l'})  = \sum_{(j', l') \rightarrow (j, l)} n_{j}(x_{j, l}).
\end{equation}

The following bound holds :

\begin{equation} \label{induction_cost}
 W_p(m_j, n_j) \leq \sum_{j' = 1}^j  2 \times 4^{-j' + 2} d \| m_j' - n_j' \|_{TV}^{1/p}
\end{equation}

\begin{proof}
 We proceed by induction on $j$. For $j = 1$, the result is obtained by using the simple bound
$W_p(m_1, n_1) \leq d \| m_1 - n_1 \|_{\text{TV}}^{1/p}$.

Suppose that (\ref{induction_cost}) holds for measures with support in $L_{j-1}$. By lemma
\ref{lemma_induction_cost_intermediate}, we have

\begin{equation*}
  W_p( m_j, n_j) \leq 2 \times 4^{-j + 2} d \| m_j - n_j\|_{TV}^{1/p} + W_p(\tilde{m}_{j-1}, \tilde{n}_{j-1})
\end{equation*}

where $\tilde{m}_{j-1}$ and $\tilde{n}_{j-1}$ are defined by (\ref{def_tilde_m}) and (\ref{def_tilde_n}) respectively.
For $1 \leq i < j-1$, define following (\ref{def_measures_above})

\begin{equation*}
 \tilde{m}_{i}(x_{i, l'})  = \sum_{(i, l') \rightarrow (j-1, l)} \tilde{m}_{j-1}(x_{j-1, l}) , \quad \tilde{n}_{i}(x_{i, l'}) = \sum_{(i, l') \rightarrow (j-1, l)} \tilde{n}_{j-1}(x_{j-1, l}).
\end{equation*}

We have

\begin{equation*}
  W_p( m_j, n_j) \leq 2 \times 4^{-j + 2} d \| m_j - n_j\|_{TV}^{1/p} + \sum_{j' = 1}^{j - 1} 2 \times 4^{-j' + 2} d \| \tilde{m}_i - \tilde{n}_i \|_{TV}^{1/p}.
\end{equation*}

To conclude, it suffices to check that for $1 \leq i \leq j - 1$,
$\| \tilde{m}_i - \tilde{n}_i \|_{TV} = \| m_i - n_i \|_{TV}$.

\end{proof}

\end{lemma}

\begin{proof}[Proof of Proposition \ref{main_prop_independent_case}.]

We pick some positive integer $k$ whose value will be determined at a later point.
Introduce the sequence of partitions $(X_{j, l})_{1 \leq l \leq m(j)}$ for $0 \leq j \leq k$ as in the lemmas above, as well as the points $x_{j, l}$.
Define $\mu_k$ as the measure with support in $L(k)$ such that $\mu_k(x_{k, l}) = \mu(X_{k, l})$ for $1 \leq l \leq m(k)$.
The diameter of the sets $X_{k, l}$ is bounded by $4^{-k + 1} d$, therefore $W_p(\mu, \mu_k) \leq 4^{-k + 1} d$.

Let $L_n^k$ denote the empirical measure associated to $\mu_k$. 

For $0 \leq j \leq k-1$, define as in Lemma \ref{lemma_induction_cost} 
the measures $\mu_j$ and $L_n^j$ with support in $L(j)$ by

\begin{align}
\mu_j (x_{j, l'}) & = \sum_{(j, l') \rightarrow (k, l)} \mu_k(x_{k, l}) \\
L_n^j (x_{j, l'}) & =  \sum_{(j, l') \rightarrow (k, l)} L_n^k(x_{k, l}).
\end{align}

It is simple to check that $\mu_j(x_{j, l}) = \mu(X_{j, l})$, and that $L_n^j$ is the empirical measure
associated with $\mu_j$.
Applying (\ref{induction_cost}), we get

\begin{equation} \label{eq_bound_empirical}
 W_p( \mu_k, L_n^k) \leq \sum_{j = 1}^k 2 \times 4^{-j + 2} d \| \mu_j - L_n^j \|_{TV}^{1/p}.
\end{equation}

Observe that $n L_n^j(x_{j, l})$ is a binomial law with parameters
$n$ and $\mu(X_{j, l})$. The expectation of $\| \mu_j - L_n^j \|_{TV}$ is bounded as follows :

\begin{align*}
 \mathbb{E} (\| \mu_i - L_n^i \|_{TV}) & = 1/2 \sum_{l = 1}^{m(j)} \mathbb{E}(| (L_n^j - \mu_j)(x_{j, l}) |) \\
{} & \leq 1/2 \sum_{l = 1}^{m(j)}  \sqrt{\mathbb{E}(| (L_n^j - \mu_j)(x_{j, l}) |^2) } \\
{} & = 1/2  \sum_{l = 1}^{m(j)} \sqrt{ \frac{ \mu(X_{j, l})(1 - \mu(X_{j, l})) } {n} } \\
{} & \leq 1/2 \sqrt{ \frac{m(j)}{n}}.
\end{align*}

In the last inequality, we use Cauchy-Schwarz's inequality and the fact that $(X_{j, l})_{1 \leq l \leq m(j)}$ is a partition
of $X$. Putting this back in (\ref{eq_bound_empirical}), we get

\begin{align*}
\mathbb{E}( W_p( \mu_k, L_n^k) ) & \leq n^{-1/2p} \sum_{j = 1}^k 2^{1-1/p}  4^{(-j + 2)} d m(i)^{1/2p} \\
{} & \leq  2^{5-1/p} n^{-1/2p} \sum_{j = 1}^k   4^{-j} d N(X, 4^{-j}d)^{1/2p} \\
{} & \leq 2^{6-1/p}/3 n^{-1/2p} \int_{4^{-(k + 1)}d}^{d/4} N(X, \delta)^{1/2p} d \delta.
\end{align*}

In the last line, we use a standard sum-integral comparison argument. 

By the triangle inequality, we have

\begin{equation*}
 W_p( \mu, L_n) \leq W_p( \mu, \mu_k) + W_p(\mu_k, L_n^k) + W_p(L_n^k, L_n).
\end{equation*}

We claim that $\mathbb{E}(W_p(L_n^k, L_n)) \leq W_p( \mu, \mu_k)$. Indeed, choose $n$ i.i.d. couples $(X_i, X_i^k)$ such that
$X_i \sim \mu$, $X_i^k \sim \mu_k$, and the joint law of $(X_i, X_i^k)$ achieves an optimal coupling, i.e.
$\mathbb{E} |X_i - X_i^k|^p = W_p^p(\mu, \mu^k)$.
We have the identities in law 

\begin{equation*}
 L_n \sim \frac{1}{n} \sum_{i = 1}^n \delta_{X_i}, \:  L_n^k \sim \frac{1}{n} \sum_{i = 1}^n \delta_{X_i^k}.
\end{equation*}

Choose the transport plan that sends $X_i$ to $X_i^k$ : this gives the upper bound 

\begin{equation*}
W_p^p(L_n, L_n^k) \leq 1/n \sum_{i = 1}^n |X_i - X_i^k|^p
\end{equation*}

and passing to expectation proves our claim.

Thus, $\mathbb{E}(W_p( \mu, L_n)) \leq 2 W_p( \mu, \mu_k) + \mathbb{E}(W_p(\mu_k, L_n^k))$.
Choose now $k$ as the largest integer such that $4^{-(k+1)}d \geq t$. This imposes $4^{-k + 1}d \leq 16 t$, and 
this finishes the proof.

\end{proof}

\begin{proof}[Proof of Corollary \ref{thm_independent_case_finite_dim}]
 It suffices to use Proposition \ref{main_prop_independent_case} along with (\ref{assumption_dimension_independent_case}) and to optimize in $t$.
\end{proof}

\section{Proof of Theorem \ref{theorem_gaussian_vector}.}

\begin{proof}[Proof of Theorem \ref{theorem_gaussian_vector}.]
 
We begin by noticing that statement (\ref{result_thm_gaussian_3}) is a simple consequence of statement
(\ref{result_thm_gaussian_2}) and the tensorization of $\bt_2$ :  we have by Corollary \ref{corollary_transport_ineq_gaussian}

\begin{equation*}
 \pr(W_2(L_n, \mu) \geq \mathbb{E}(W_2(L_n, \mu) + t) \leq e^{-n t^2/(2 \sigma^2)},
\end{equation*}

and it suffices to choose $t = \lambda \psi^{-1}( \log n)$ to conclude. We now turn to the other claims.

Denote by $K$ the unit ball of the Cameron-Martin space associated to $E$ and $\mu$, and by
$B$ the unit ball of $E$.
According to the Gaussian isoperimetric inequality (see \cite{ledoux1996isoperimetry}),
for all $\lambda > 0$ and $\varepsilon > 0$,

\begin{equation*}
\mu( \lambda K + \varepsilon B) \geq \Phi \left( \lambda + \Phi^{-1}( \mu( \varepsilon B)) \right)
\end{equation*}

where $\Phi(t) = \int_{- \infty}^t e^{- u^2/2} d u / \sqrt{2 \pi}$ is the Gaussian c.d.f..

Choose $\lambda > 0$ and $\varepsilon > 0$, and set $X = \lambda K + \varepsilon B$. Note 

\begin{equation*}
 \mu' = \frac{1}{ \mu( X)} \mathbf{1}_{X} \mu
\end{equation*}

the restriction of $\mu$ to the enlarged ball. 

The diameter of $X$ is bounded by $2( \sigma \lambda + \varepsilon)$. 
The $W_2$ distance between $L_n$ and $\mu$ is thus bounded as follows :

\begin{equation} \label{bound_w_2_gaussian}
 W_2(L_n, \mu) \leq 2 W_2( \mu, \mu') + c t + c n^{-1/4} \int_{t}^{( \sigma \lambda + \varepsilon)/2} N(X, \delta)^{1/4} d \delta
\end{equation}

Set 

\begin{align}
 I_1 & = W_2(\mu, \mu') \\
 I_2 & = t \\
 I_3 & = n^{-1/4} \int_{t}^{( \sigma \lambda + \varepsilon)/2} N(X, \delta)^{1/4} d \delta.
\end{align}

To begin with, set $\varepsilon = t/2$.

\emph{Controlling $I_1$.} We use transportation inequalities and the
Gaussian isoperimetric inequality. By Lemma \ref{lemma_transport_ineq_gaussian}, $\mu$ satisfies a $\bt_2(2 \sigma^2)$ inequality, 
so that we have

\begin{align*}
 W_2(\mu, \mu') & \leq \sqrt{ 2 \sigma^2 H( \mu' | \mu)} = \sqrt{-2 \sigma^2 \log \mu( \lambda K + \varepsilon B)} \\
{} & \leq \sqrt{-2 \sigma^2 \log \Phi( \lambda  + \Phi^{-1}(\mu(\varepsilon B)))} \\
{} & = \sqrt{2} \sigma \sqrt{ - \log \Phi ( \lambda + \Phi^{-1}(e^{-\psi(t/2)})) }.
\end{align*}

Introduce the tail function of the Gaussian distribution 

\begin{equation*}
\Upsilon(x) = \sqrt{2 \pi}^{-1} \int_x^{ + \infty} e^{-y^2/2} d y. 
\end{equation*}

We will use the fact that $\Phi^{-1} + \Upsilon^{-1} = 0$, which comes from symmetry of the Gaussian distribution. We will also use the bound
$\Upsilon(t) \leq e^{-t^2/2}/2$, $t \geq 0$ and its consequence

\begin{equation*}
 \Upsilon^{-1}(u) \leq \sqrt{-2 \log u}, \quad 0 < u \leq 1/2.
\end{equation*}

We have

\begin{equation*}
 \Phi^{-1}(e^{-\psi(t/2)}) = - \Upsilon^{-1}(e^{-\psi(t/2)}) \geq - \sqrt{2 \psi(t/2)}
\end{equation*}

as soon as $\psi(t/2) \geq \log 2$.
The elementary bound $\log \frac{1}{1 - x} \leq 2x$ for $x \leq 1/2$ yields

\begin{align*}
 \sqrt{- 2 \log \Phi ( u)} & = \sqrt{2} \left( \log \frac{1}{ 1 - \Upsilon(u)} \right)^{1/2} \\
{} & \leq \sqrt{2} e^{-u^2/4}
\end{align*}

whenever $u \geq \Upsilon^{-1}(1/2) = 0$. Putting this together, we have

\begin{equation} \label{control_I_1}
 I_1 \leq \sqrt{2} \sigma e^{- (\lambda - \sqrt{2 \psi(t/2)})^2/4}.
\end{equation}

whenever 

\begin{equation} \label{technical_conditions_gaussian_measures}
\psi(t/2) \geq \log 2 \text{ and }\lambda - \sqrt{2 \psi(t/2)} \geq 0.
\end{equation}

\emph{Controlling $I_3$.} The term $I_3$ is bounded by $1/2 n^{-1/4} (\sigma \lambda + t/2) N(X, t)^{1/4}$ (just bound
the function inside by its value at $t$, which is minimal).
Denote $k = N(\lambda K, t - \varepsilon)$ the covering number of $\lambda K$
(w.r.t. the norm of $E$).
Let $x_1, \ldots, x_k \in K$ be such that union of the balls
$B(x_i, t - \varepsilon)$ contains $\lambda K$.
From the triangle inequality we get the inclusion

\begin{equation*}
 \lambda K + \varepsilon B \subset \bigcup_{i = 1}^k B(x_i, t).
\end{equation*}

Therefore, $N(X, t) \leq N(\lambda K, t - \varepsilon) = N(\lambda K, t/2)$.

We now use the well-known link between $N(\lambda K, t/2)$ and the small ball function.
 Lemma 1 in \cite{kuelbs1993metric} gives the bound

\begin{equation*} 
 N( \lambda K, t/2) \leq e^{\lambda^2/2 + \psi(t/4)} \leq e^{\lambda^2/2 + \kappa \psi(t/2)}.
\end{equation*}

so that 

\begin{equation} \label{control_I_3}
 I_3 \leq \frac{1}{2} (\sigma \lambda + t/2) e^{ \frac{\lambda^2}{8} + \frac{\kappa}{4}  \psi(t/2) - \frac{1}{4} \log n }.
\end{equation}

Remark that we have used the doubling condition on $\psi$, so that we require 

\begin{equation} \label{technical_conds_gaussian_2}
 t/4 \leq t_0.
\end{equation}

\emph{Final step.}  Set now $t = 2 \psi^{-1}( a \log n)$ and $\lambda = 2 \sqrt{2 a \log n}$, with $a > 0$ yet undetermined. Using (\ref{control_I_1}) and (\ref{control_I_3}), we see that
there exists a universal constant $c$ such that

\begin{align*}
 \mathbb{E}(W_2( L_n, \mu)) \leq & c \left[ \psi^{-1}( a \log n) + \sigma e^{- (a/2) \log n} \right.\\
{} & \quad \left. + (\sigma \sqrt{a \log n} + \psi^{-1}( a \log n) ) e^{[a(1 + \kappa/4) - 1/4] \log n} \right].
\end{align*}

Choose $a = 1/(6 + \kappa)$ and assume 
$\log n \geq (6 + \kappa) (\log 2 \vee \psi(1) \vee \psi(t_0/2) )$.
This guarantees that the technical conditions 
(\ref{technical_conditions_gaussian_measures}) and (\ref{technical_conds_gaussian_2}) are enforced, and that 
$\psi^{-1}( a \log n) \leq 1$. Summing up, we get :

\begin{equation*}
 \mathbb{E}(W_2( L_n, \mu)) \leq c \left[  \psi^{-1}( \frac{1}{6 + \kappa} \log n) + (1 + \sigma \sqrt{\frac{1}{6 + \kappa} \log n}) n^{- 1/(12 + 2 \kappa)} \right].
\end{equation*}

Impose $\log n \geq (6 + \kappa)/\sigma^2$ : this ensures $\sigma \sqrt{ \frac{1}{6 + \kappa} \log n} \geq 1$. And finally, there exists some $c > 0$ such that
for all $x \geq 1$, $\sqrt{ \log x} x^{-1/4} \leq c$ : this implies 

\begin{equation*}
 \sqrt{\frac{1}{6 + \kappa} \log n} n^{-1/(24 + 4\kappa)} \leq c.
\end{equation*}

This gives

\begin{equation*}
 (1 + \sigma \sqrt{\frac{1}{6 + \kappa} \log n}) n^{- 1/(12 + 2 \kappa)} \leq c \sigma n^{-1/[4(6 + \kappa)]}
\end{equation*}

and the proof is finished.

\end{proof}

\section{Proofs in the dependent case}

We consider hereafter a Markov chain $(X_n)_{n \in \mathbb{N}}$ defined by $X_0 \sim \nu$ and the transition kernel $P$.
Let us denote by

\begin{equation*}
 L_n = \sum_{i = 1}^n \delta_{X_i}
\end{equation*}

its occupation measure.

\begin{prop}

Suppose that the Markov chain satisfies (\ref{decay_of_variance}) for some $C > 0$ and $\lambda < 1$. Then the following holds :

\begin{equation} \label{bound_markov_poincare}
 \mathbb{E}_\nu (W_p (L_n, \pi)) \leq c \left( t + \left( \frac{C}{(1 - \lambda) n} \|\frac{d \nu}{d \pi} \|_{r} \right)^{1/2p} \int_t^{d/4} N(X, t)^{1/2p(1+1/r)} d t \right).
\end{equation}

\end{prop}

\begin{proof}

An application of (\ref{induction_cost}) as in 
(\ref{eq_bound_empirical}) yields

\begin{equation} \label{eq_bound_markov}
 \mathbb{E}( W_p (L_n, \pi) ) \leq 2 \times 4^{-k + 1}d + \sum_{j = 1}^k 2 \times 4^{-j + 2} d \left( \sum_{l = 1}^{m(j)} \mathbb{E}|(L_n - \pi)(X_{j, l})| \right)^{1/p}.
\end{equation}

Let $A$ be a measurable subset of $X$, and
set $f_A(x) = \mathbf{1}_A (x) - \pi (A)$. We have

\begin{align*}
\mathbb{E}|(L_n - \pi)(A)| & = 1/n \mathbb{E}_\nu | \sum_{i = 1}^n f_A (X_i)|  \\
& \leq 1/n \sqrt{ \sum_{i = 1}^n \sum_{j = 1}^n \mathbb{E}_\nu \left[ f_A (X_i) f_A(X_j) \right] }.
\end{align*}

Let $\tilde{p}, \tilde{q}, r \geq 1$ be such that $1/\tilde{p} + 1/\tilde{q} + 1/r = 1$, and let $s$ be defined by $1/s = 1/\tilde{p} + 1/\tilde{q}$.
Now, using H\"older's inequality with $r$ and $s$,

\begin{equation*}
 \mathbb{E}_\nu \left[ f_A (X_i) f_A(X_j) \right] \leq \|\frac{d \nu}{d \pi} \|_{r} (\mathbb{E}_\pi | f_A (X_i) f_A(X_j) |^s  )^{1/s}.
\end{equation*}

Use the Markov property and the fact that $f \mapsto Pf$ is a contraction in $L^s$ to get

\begin{equation*}
 \mathbb{E}_\nu \left[ f_A (X_i) f_A(X_j) \right] \leq \|\frac{d \nu}{d \pi} \|_{r} \| f_A P^{j-i} f_A \|_s.
\end{equation*}

Finally, use H\"older's inequality with $\tilde{p}, \tilde{q}$ : we get

\begin{equation}  \label{bound_occupation_measure_r_s}
\mathbb{E}_\nu \left[ f_A (X_i) f_A(X_j) \right] \leq \|\frac{d \nu}{d \pi} \|_{r} \| P^{j-i} f_A \|_{\tilde{p}} \| f_A \|_{\tilde{q}}.
\end{equation}

Set $\tilde{p} = 2$ and note that for $1 \leq t \leq + \infty$, we have $\|f_A\|_{t} \leq 2 \pi(A)^{1/t}$. 
Use (\ref{decay_of_variance}) applied to the centered function $f_A$ to get

\begin{equation*}
  \mathbb{E}_\nu \left[ f_A (X_i) f_A(X_j) \right] \leq  4 C \lambda^{j-i} \|\frac{d \nu}{d \pi} \|_{r} \pi(A)^{1 - 1/r},
\end{equation*}

and as a consequence,

\begin{equation} \label{markov_chain_occupation_time}
 \mathbb{E}|(L_n - \pi)(A)| \leq \frac{1}{\sqrt{n}} \frac{ 2 \sqrt{2 C} }{ \sqrt{1 - \lambda} } \|\frac{d \nu}{d \pi} \|_{r}^{1/2} \pi(A)^{1/2 - 1/2r}.
\end{equation}

Come back to (\ref{eq_bound_markov}) : we have

\begin{align*}
  \mathbb{E}( W_p (L_n, \pi) ) & \leq 4^{-k + 1}d + 32 (\frac{ 2 \sqrt{2 C}}{ \sqrt{1 - \lambda} })^{1/p} \|\frac{d \nu}{d \pi} \|_{r}^{1/2p} n^{-1/2p} \\
{} & \quad \times \sum_{j = 1}^k 4^{-j} d \left( \sum_{l = 1}^{m(j)} \pi(X_{j, l})^{1/2 - 1/2r} \right)^{1/p} \\
{} & \leq 4^{-k + 1}d + c \left( \frac{C}{(1 - \lambda) n} \|\frac{d \nu}{d \pi} \|_{r} \right)^{1/2p} \sum_{j = 1}^k 4^{-j} d m(j)^{1/2p(1+1/r)} \\
{} & \leq c \left( t + \left( \frac{C}{(1 - \lambda) n} \|\frac{d \nu}{d \pi} \|_{r} \right)^{1/2p} \int_t^{d/4} N(X, t)^{1/2p(1+1/r)} d t \right).
\end{align*}

\end{proof}

\begin{proof}[Proof of Theorem \ref{thm_markov}.]
Use (\ref{bound_markov_poincare}) and (\ref{dimension_assumption_markov}) to get

\begin{equation*}
 \mathbb{E} W_p(L_n, \mu) \leq c \left[ t + A  t^{-\alpha/2p(1+1/r) + 1} \right]
\end{equation*}

where 

\begin{equation*}
A = \frac{2p}{\alpha(1+1/r)} (C/(1 - \lambda))^{1/2p}\|\frac{d \nu}{d \pi} \|_{r}^{1/2p} n^{-1/2p} d^{\alpha/2p(1+1/r)}.
\end{equation*}

Optimizing in $t$ finishes the proof.

\end{proof}

We now move to the proof in the unbounded case.

\begin{proof}[Proof of Theorem \ref{thm_markov_unbounded}]

We remind the reader that the following assumption stands~: for $X \subset E$ with diameter bounded by $d$,

\begin{equation} \label{assumption_state_space_markov}
 N(X, \delta) \leq k_E (d/\delta)^\alpha.
\end{equation}

In the following lines, we will make use of the elementary inequalities

\begin{equation} \label{elementary_ineq}
 (x + y)^p \leq 2^{p-1} (x^p + y^p) \leq 2^{p - 1} (x + y)^p.
\end{equation}

\emph{Step 1.}

Pick increasing sequence of numbers $d_i > 0$ to be set later on, and some point $x_0 \in E$.
Define $C_1 = B(x_0, d_1)$, and $C_i = B(x_0, d_i) \ B(x_0, d_{i-1})$ 
for $i \geq 2$.

The idea is as follows : we decompose the state space $E$ into a union of rings, and deal separately with $C_1$ on the one hand, using the case of Theorem \ref{thm_markov} as guideline, 
and with the union of the $C_i$, $i \geq 2$ on the other hand, where we use more brutal bounds.

We define partial occupation measures 

\begin{equation*}
 L_n^i = 1/n \sum_{j = 1}^n \delta_{X_j} \mathbf{1}_{X_j \in C_i}
\end{equation*}

and their masses $m_i = L_n^i(E)$. We have the inequality

\begin{equation} \label{decomposition_w_p}
 W_p^p(L_n, \pi) \leq \sum_{i \geq 1} m_i W_p^p( 1/m_i L_n^i, \pi).
\end{equation}

On the other hand,

\begin{align*}
 W_p(1/m_i L_n^i, \pi) & \leq (\int d(x_0, x)^p d \pi)^{1/p} + (\int d(x_0, x)^p d (1/m_i L_n^i) )^{1/p} \\
{} & \leq M_p^{1/p} + d_i,
\end{align*}

so that $W_p^p(1/m_i L_n^i, \pi) \leq 2^{p-1} \left( M_p + d_i^p \right)$ using (\ref{elementary_ineq}). Also, using (\ref{decomposition_w_p}) and
(\ref{elementary_ineq}) yields

\begin{equation*}
W_p(L_n, \pi) \leq m_1^{1/p} W_p(1/m_1 L_n^1, \pi) + 2^{1 - 1/p} \left( \sum_{i \geq 2} m_i \left[ M_p + d_i^p \right] \right)^{1/p}.
\end{equation*}

Pass to expectations to get

\begin{equation} \label{bound_decomposition}
 \mathbb{E} [ W_p(L_n, \pi) ] \leq \mathbb{E} \left[ m_1^{1/p} W_p(1/m_1 L_n^1, \pi) \right] + 2^{1 - 1/p} \left( \sum_{i \geq 2} \pi(C_i) \left[ M_p + d_i^p \right] \right)^{1/p}
\end{equation}

We bound separately the left and right term in the right-hand side of (\ref{bound_decomposition}), starting with the right one.

\emph{Step 2}.

Choose some $q > p$ and use Chebyshev's inequality to bound the sum on the right by

\begin{equation} \label{bound_decomposition_sum}
 \sum_{i \geq 2} \frac{M_q}{d_{i-1}^q} \left[ M_p + d_i^p \right]
\end{equation}

Take $d_i = \rho^i M_p^{1/p}$, (\ref{bound_decomposition_sum}) becomes

\begin{align*}
{} &  M_q M_p^{1 - q/p} \rho^{q} \sum_{i \geq 2} [ \rho^{-q i} + \rho^{(p - q)i} ] \\
=  &    M_q M_p^{1 - q/p} \left[ \frac{\rho^{-q}}{1 - \rho^{-q}} + \frac{\rho^{2p - q}}{1 - \rho^{p - q}} \right].
\end{align*}

Assume for example that $\rho \geq 2$ : this implies

\begin{equation*} 
 \sum_{i \geq 2} \pi(C_i) \left[ M_p + d_i^p \right] \leq 4 M_q M_p^{1 - q/p} \rho^{2p-q}.
\end{equation*}

For later use, we set $\zeta = q/p - 2$ and the above yields

\begin{equation} \label{bound_decomposition_sum_2}
2^{1 - 1/p} \left( \sum_{i \geq 2} \pi(C_i) \left[ M_p + d_i^p \right] \right)^{1/p} \leq 4 M_{(\zeta + 2)p}^{1/p} M_p^{-(1 + \zeta)/p} \rho^{-\zeta}. 
\end{equation}

\emph{Step 3.}

We now turn our attention to the term on the left in (\ref{bound_decomposition}). 

Once again, we apply (\ref{induction_cost}) to obtain

\begin{equation*}
 W_p(1/m_1 L_n^1, \pi) \lesssim 4^{-k} d_1 + \sum_{j = 1}^k 4^{-j} \left( \sum_{l = 1}^{m(j)} | ((1/m_1)L_n - \pi) (X_{j, l}) | \right)^{1/p}
\end{equation*}

Multiply by $m_1^{1/p}$ and pass to expectations :

\begin{align*}
 \mathbb{E} \left[ m_1^{1/p} W_p(m_1^{-1} L_n^1, \pi) \right] \lesssim & \, \sum_{j = 1}^k 4^{-j} \left( \sum_{l = 1}^{m(j)} \mathbb{E} | (L_n - m_1 \pi) (X_{j, l}) | \right)^{1/p} \\
{} & + 4^{-k} d_1 \mathbb{E} (m_1^{1/p}).
\end{align*}

First, notice that $0 \leq m_1 \leq 1$ a.s. so that $\mathbb{E} (m_1^{1/p}) \leq 1$. Next, write

\begin{align*}
 \sum_{l = 1}^{m(j)} \mathbb{E} | (L_n - m_1 \pi) (X_{j, l}) | & \leq \sum_{l = 1}^{m(j)} \mathbb{E} \left( | (L_n - \pi) (X_{j, l}) | + | (m_1 \pi - \pi)(X_j,l) |\right) \\
{} & \leq \sum_{l = 1}^{m(j)} \mathbb{E} | (L_n - \pi) (X_{j, l}) | + \mathbb{E}(|m_1 - 1|) \pi(C_1) \\
{} & \leq \sum_{l = 1}^{m(j)} \mathbb{E} | (L_n - \pi) (X_{j, l})| + \mathbb{E} |L_n(C_1) - 1|.
\end{align*}

The first of these two terms is controlled using (\ref{markov_chain_occupation_time}) : we have

\begin{equation*}
 \sum_{l = 1}^{m(j)} \mathbb{E} | (L_n - \pi) (X_{j, l})| \leq \frac{1}{\sqrt{n}} \frac{ 2 \sqrt{2 C} }{ \sqrt{1 - \lambda} } \|\frac{d \nu}{d \pi} \|_{r}^{1/2} m(j)^{1/2 + 1/2r}
\end{equation*}

And on the other hand,

\begin{align*}
 \mathbb{E} |L_n(C_1) - 1| & \leq \mathbb{E} |(L_n - \pi)(C_1)| + \pi(C_1^c) \\
{} & \leq \frac{1}{\sqrt{n}} \frac{ 2 \sqrt{2 C} }{ \sqrt{1 - \lambda} } \|\frac{d \nu}{d \pi} \|_{r}^{1/2} + \pi(C_1^c).
\end{align*}

Here we have used (\ref{markov_chain_occupation_time}) again.

We skip over details here as they are similar to those in previous proofs. Choosing an appropriate value for $k$ and using the estimates above
allows us to recover the following :

\begin{align} \label{bound_decomposition_sum_3}
 \mathbb{E} \left[ m_1^{1/p} W_p(1/m_1 L_n^1, \pi) \right] \lesssim & \left( \frac{C}{(1 - \lambda) n} \|\frac{d \nu}{d \pi} \|_{r} \right)^{1/2p} \int_t^{d_1/4} N(C_1, \delta)^{1/2p(1 + 1/r)}  d \delta \\
\nonumber {} & + \pi(C_1^c)   + t.
\end{align}

The term $\pi(C_1^c)$ is bounded by the Chebyshev inequality :

\begin{equation*}
 \pi(C_1^c) \leq \int x^\zeta d \pi/ d_1^\zeta = \int x^\zeta d \pi \left(\int x^p d \pi \right)^{- \zeta/p} \rho^{-\zeta}.
\end{equation*}

\emph{Step 4.}

Use (\ref{bound_decomposition_sum_2}) and (\ref{bound_decomposition_sum_3}), along with assumption (\ref{assumption_state_space_markov}) : this yields

\begin{equation*}
 \mathbb{E} (W_p(L_n, \pi)) \lesssim K(\zeta) \left( \rho^{- \zeta} + t +  A_n \rho^{\alpha/2p(1 + 1/r)} t^{1-\alpha/2p(1 + 1/r)} \right)
\end{equation*}

where $A_n = \left( \frac{C}{(1 - \lambda) n} \|\frac{d \nu}{d \pi} \|_{r} \right)^{1/2p}$, and

\begin{equation*}
K(\zeta) = \frac{m_\zeta}{m_p^{\zeta/p}} \vee \frac{m_{\zeta + 2p}}{m_p^{1 + \zeta/p}} \vee k_E^{1/2p(1+1/r)} \frac{2p}{\alpha(1 + 1/r)} m_p^{\alpha/(2p^2)(1+1/r)}.
\end{equation*}

The remaining step is optimization in $t$ and $\rho$. We obtain the following result : there exists a constant $C(p, r, \zeta)$ depending only on the values of $p$, $r$, $\zeta)$, such that

\begin{equation*}
 \mathbb{E} (W_p(L_n, \pi)) \lesssim C(p, r, \zeta)  K(\zeta) A_n^{2p/(\alpha(1 + 1/r)(1 + 1/\zeta))}. 
\end{equation*}

There is a caveat : we have used the condition $\rho \geq 2$ at some point, and with this restriction the optimization above is valid only when $A_n \leq C'(p, r, \zeta)$, where the constant
$C'(p, r, \zeta)$ only depends on the values of $p$, $r$, $\zeta$.

\end{proof}

\appendix

\section{Transportation inequalities for Gaussian measures on a Banach space} \label{appendix_gaussian}

Transportation inequalities, also called transportation-entropy inequalities, have been introduced by K. Marton \cite{marton1996bounding}
to study the phenomenon of concentration of measure. M. Talagrand showed that the finite-dimensional Gaussian measures satisfy a $\bt_2$ inequality.
The following appendix contains a simple extension of this result to the infinite-dimensional case. For much more on the topic of transportation inequalities,
the reader may refer to the survey \cite{gozlan2010transport} by N. Gozlan and C. L\'eonard.

For $\mu \in \Pro(E)$, let $H(.|\mu)$ denote the relative entropy with respect to $\mu$ :

\begin{equation*}
 H( \nu | \mu) = \int_E \frac{d \nu} { d \mu} \log \frac{d \nu}{ d \mu} d \mu
\end{equation*}

if $\nu \ll \mu$, and $H( \nu | \mu) = + \infty$ otherwise.

We say that $\mu \in \Pro_p(E)$ satisfies a $\bt_p(C)$ transportation inequality when 

\begin{equation*}
 W_p( \nu, \mu) \leq \sqrt{C H( \nu | \mu)} \quad \forall \nu \in \Pro_p(E)
\end{equation*}

We identify what kind of transport inequality is satisfied by a Gaussian measure on
a Banach space. We remind the reader of the following definition :
let $(E, \mu)$ be a Gaussian Banach space and $X \sim \mu$ be a $E$-valued r.v.. The weak variance
of $\mu$ or $X$ is defined by

\begin{equation*}
 \sigma^2 = \sup_{f \in E^*, |f| \leq 1} \mathbb{E}( f^2(X)).
\end{equation*}

The lemma below is optimal, as shown by the finite-dimensional case.

\begin{lemma} \label{lemma_transport_ineq_gaussian}
 Let $(E, \mu)$ be a Gaussian Banach space, and let $\sigma^2$ denote the weak variance of $\mu$.
Then $\mu$ satisfies a $\bt_2(2 \sigma^2)$ inequality.
\end{lemma}

\begin{proof}
According e.g. to \cite{ledoux1991probability}, there exists a sequence $(x_i)_{i \geq 1}$ in
$E$ and an orthogaussian sequence $(g_i)_{i \geq 1}$ (meaning a sequence of i.i.d. standard normal variables)
such that

\begin{equation*}
 \sum_{i \geq 1} g_i x_i \sim \mu,
\end{equation*}

where convergence of the series holds a.s. and in all the $L^p$'s. In particular, the laws $\mu_n$
of the partial sums $\sum_{i = 1}^n g_i x_i$ converge weakly to $\mu$.

As a consequence of the stability result of Djellout-Guillin-Wu (Lemma 2.2 in \cite{djellout_guillin_wu}) showing that
$\bt_2$ is stable under weak convergence, it thus suffices to show that the measures $\mu_n$ all satisfy
the $\bt_2(2 \sigma^2)$ inequality.

First, by definition of $\sigma$, we have

\begin{equation*}
 \sigma^2 = \sup_{f \in E^*, |f| \leq 1} \mathbb{E} (\sum_{i = 1}^{+ \infty}  f(x_i) g_i)^2
\end{equation*}

and since $(g_i)$ is an orthogaussian sequence, the sum is equal to $\sum_{i = 1}^{+ \infty} f^2(x_i)$.

Consider the mapping 

\begin{align*}
 T : & (\R^n, N) \rightarrow (E, \|.\|) \\
{} & (a_1, \ldots, a_n) \mapsto \sum_{i = 1}^n a_i x_i.
\end{align*}

(here
$\R^n$ is equipped with the Euclidean norm $N$). With the remark above it is easy to check that 
$\|T(a)\| \leq \sigma N(a)$ for $a \in \R^n$. Consequently, $T$ is $\sigma$-Lipschitz, and
we can use the second stability result of Djellout-Guillin-Wu (Lemma 2.1 in \cite{djellout_guillin_wu}) :
the push forward of a measure satisfying $\bt_2(C)$ by a $L$-Lipschitz function satisfies $\bt_2(L^2C)$.
As is well-known, the standard Gaussian measure $\gamma^n$ on $\R^n$ satisfies $\bt_2(2)$ and thus
$T_\# \gamma^n$ satisfies $\bt_2(2 \sigma^2)$. But it is readily checked that
$T_\# \gamma^n = \mu_n$, which concludes this proof.

\end{proof}

\begin{remark}
 M.Ledoux indicated to us another way to obtain this result. First, one shows that the Gaussian measure satisfies a $\bt_2(2)$ inequality when considering the cost 
function $c = d_H^2$, where $d_H$ denotes the Cameron-Martin metric on $E$ inherited from the scalar product on the Cameron-Martin space. This can be done in a number of ways,
for example by tensorization of the finite-dimensional $\bt_2$ inequality for Gaussian measures or by adapting the Hamilton-Jacobi arguments of Bobkov-Gentil-Ledoux
 \cite{bobkov2001hypercontractivity} in the infinite-dimensional setting. It then suffices to observe that this transport inequality implies the 
one we are looking for since we have the bound $d \leq \sigma d_H$ (here $d$ denotes the metric inherited from the norm of the Banach space).
\end{remark}

Let $L_n$ denote the empirical measure associated with $\mu$.
As a consequence of Lemma \ref{lemma_transport_ineq_gaussian}, we can give an inequality for the concentration of $W_2(L_n, \mu)$ around its mean, using results from
transportation inequalities. This is acutally a simple case of more general  results of N. Gozlan and C. L\'eonard (\cite{large_dev_gozlan_leonard}, \cite{gozlan2010transport}),
 we reproduce a proof here for convenience.

\begin{corollary} \label{corollary_transport_ineq_gaussian}
 Let $\mu$ be as above. The following holds :

\begin{equation*}
 \pr( W_2(L_n, \mu) \geq \mathbb{E} [ W_2(L_n, \mu) ] + t) \leq e^{- n t^2/(2 \sigma^2)}.
\end{equation*}

\end{corollary}

\begin{proof}
 The proof relies on the property of dimension-free tensorization of the $\bt_2$ inequality, see \cite{gozlan2010transport}.
Since $\mu$ satisfies $\bt_2(2 \sigma^2)$, the product measure $\mu^{ \otimes n}$ on the product space $E^n$ endowed with the $l_2$
metric 

\begin{equation*}
 d_2 ( (x_1, \ldots, x_n), (y_1, \ldots, y_n)) = \sqrt{ |x_1 - y_1|^2 + \ldots + |x_n - y_n|^2}
\end{equation*}

also satisfies a $\bt_2( 2 \sigma^2)$ inequality (\cite{gozlan2010transport}, Corollary 4.4). Therefore, it also satisfies a $\bt_1$ inequality by Jensen's inequality, and
this implies that we have the concentration inequality

\begin{equation*}
 \mu^{\otimes n} ( f \geq \int f d \mu^{\otimes n} + t) \leq e^{- t^2/(2 \sigma^2)}
\end{equation*}

for all $1$-Lipschitz functions $f : (E^n, d_2) \rightarrow \R$ (\cite{gozlan2010transport}, Theorem 1.7).
For $x = (x_1, \ldots, x_n) \in E^n$, denote $L_n^x = 1/n \sum_{i = 1}^n \delta x_i$. 
To conclude it suffices to notice that $(x_1, \ldots, x_n) \rightarrow W_2( L_n^x, \mu)$ is $\sqrt{n}$-Lipschitz from $(E^n, d_2)$ to $\R$.
\end{proof}

\bibliography{exact_deviations}{}
\bibliographystyle{plain}

\end{document}